\documentclass[12pt]{article}
\usepackage{amsmath,amssymb,amsthm,amsfonts,graphicx,color,enumerate,float,epsfig,comment}
\usepackage{tikz}
\begin{document}

\newcommand{\kf}[1]{\marginpar{\color{red}\tiny #1 --kf}}
\newtheorem{thm}{Theorem}[section]
\newtheorem{lemma}[thm]{Lemma}
\newtheorem{cor}[thm]{Corollary}
\newtheorem{prop}[thm]{Proposition}
\newtheorem{question}[thm]{Question}
\newtheorem*{mainthm0}{Theorem \ref{log}}
\newtheorem*{main}{Main Theorem}

\theoremstyle{remark}
\newtheorem{remark}[thm]{Remark}
\newtheorem{important remark}[thm]{Important Remark}
\newtheorem{definition}[thm]{Definition}
\newtheorem{q}[thm]{}
\newtheorem{example}[thm]{Example}
\newtheorem{fact}[thm]{Fact}
\newtheorem{convention}[thm]{Convention}

\def\diam{\operatorname{diam}}
\def\stab{\operatorname{stab}}
\def\supp{\operatorname{supp}}
\def\tr{\operatorname{tr}}
\def\cal{\mathcal}
\def\R{{\mathbb R}}
\def\ZZ{{\mathbb Z}}
\def\Q{{\mathbb Q}}
\def\G{\Gamma}

\def\Cayley{{\rm Cayley}}
\def\Mod{{\rm Mod}}
\def\ep{{\varepsilon}}
\def\Isom{{\rm Isom}}
\def\C{{\mathcal C}}
\def\T{{\mathcal T}}
\def\TT{\overline{\mathcal T}}
\def\dTT{\partial \overline{\mathcal T}}
\def\S{{\mathcal S}}
\def\SS{{\overline{\mathcal S}}}
\def\M{{\mathcal M}}
\def\L{{\mathcal L}}
\def\D{\Delta}
\def\s{\sigma}
\def\t{\tau}
\def\<{{\langle}}
\def\>{{\rangle}}

\def\bY{\bold Y}
\def\cP{\cal P}
\def\cC{\cal C}
\def\N{{\mathcal N}}

\def\red{\textcolor{red}}
\def\CG{{\mathcal G}}
\def\CA{{\mathcal {AX}}}
\newcommand{\A}{{\mathcal A}}
\newcommand{\B}{{\mathcal B}}

\title{Proper actions on finite products of quasi-trees}
\date{\today}

\author{Mladen Bestvina, Ken Bromberg, Koji Fujiwara
  \thanks{The first two authors gratefully acknowledge the
    support by the National Science Foundation. The third author is
    supported in part by Grant-in-Aid for Scientific Research
    (No. 15H05739, 20H00114)}}

\newcommand{\mb}[1]{\marginpar{\color{blue}\tiny #1 --mb}}
\newcommand{\kb}[1]{\marginpar{\color{green}\tiny #1 --kb}}

\maketitle

\abstract{We say that a finitely generated group $G$ has
property (QT) if it acts isometrically on a finite product of
quasi-trees so that orbit maps are quasi-isometric embeddings. A
quasi-tree is a connected graph with path metric quasi-isometric to a
tree, and product spaces are equipped with the
$\ell^1$-metric. 

We prove that residually finite hyperbolic groups and mapping class groups have (QT). }

\section{Introduction} We say that a finitely generated group $G$ has
property (QT) if it acts isometrically on a finite product of
quasi-trees so that orbit maps are quasi-isometric embeddings. A
quasi-tree is a connected graph with path metric quasi-isometric to a
tree, and product spaces are equipped with the
$\ell^1$-metric. The
first examples of such groups come with proper actions on products of
trees, for example free groups, surface groups (e.g. take the product
of Bass-Serre trees dual to a finite collection of filling curves), or
products thereof. In \cite{dj} Dranishnikov and Januszkiewicz show
that any Coxeter group admits such an action on a finite product of
trees. In particular, the same is true for any undistorted finitely
generated subgroup, and also for any commensurable group (see below),
and in particular it holds for right angled Artin groups.
It then also follows for the Haglund-Wise virtually special groups, since these are commensurable to finitely generated undistorted subgroups of RAAGs.

The goal of this paper is to prove the following two theorems
as an application of the projection complex techniques developed
in \cite{bbf}; see also \cite{bbfs}.

\begin{thm}\label{thm1}
  Let $G$ be a residually finite hyperbolic group. Then
  $G$ has (QT).
\end{thm}

\begin{thm}\label{thm2}
  Mapping class groups have (QT).
\end{thm}

Hamenst\"adt announced Theorem \ref{thm2} in the Fall 2016, but
our proof is different. Earlier, Hume \cite{hume} constructed a
(nonequivariant) quasi-isometric embedding of mapping class groups in a
finite products of trees. In the Spring 2018 Hamenst\"adt also announced
that Artin groups of finite type have (QT).

Cocompact lattices in $Sp(n,1), n>1$  satisfy the assumptions of
Theorem \ref{thm1} and they have Kazhdan's property (T). 
In particular they do not have the Haagerup property, namely,
they do not act properly by isometries on the Hilbert space.
Recall also that
if  a group with property (T)
acts on a  tree, then it must have a fixed point
(by Serre and Watatani, cf. \cite[Section 2.3]{V}).
On the other hand if a finitely generated 
group acts properly on a finite dimensional CAT(0) cube complex, 
e.g., a finite product of  simplicial trees, then it has the Haagerup
property, \cite{NR}.

In view of the lattice example in $Sp(n,1)$, by Theorem \ref{thm1},
having a proper action on a finite product of quasi-trees that gives a
quasi-isometric embedding of a group is not enough to expect a proper
isometric group action on the Hilbert space.
It is unknown if mapping class groups have
either property (T) or the Haargerup property.

Property (QT) is a strong form of finiteness of asymptotic dimension.
It was proved by Gromov \cite{gromov-asymptotic} that hyperbolic
groups have finite asymptotic dimension, and by the authors in
\cite{bbf} that mapping class groups do as well. See also \cite{bhs}
for a quadratic bound.
Also, it was known that a hyperbolic group
admits a quasi-isometric embedding into the 
product of $n+1$ binary trees, where $n$ is the topological 
dimension of the boundary at infinity of the group, \cite{BDS}.

Finally, we remark that higher rank lattices do not have (QT) even
though they have finite asymptotic dimension: by
\cite{derham} an isometric action on a finite product of quasi-trees
preserves the deRham decomposition, and by Haettel \cite{haettel}
higher rank lattices do not have non-elementary actions on
quasi-trees. 

We would like to thank the referee for comments, which 
improved the presentation of the paper.

\section{Background with complements}

\subsection{Separability}
We thank Chris Leininger and Ben McReynolds for pointing out the
following fact. Recall that a subgroup $H<G$ is {\it separable} if it
is the intersection of all finite index subgroups that contain it. Thus
$G$ is residually finite if and only if the trivial subgroup is
separable. 

\begin{lemma}\label{separability}
  Suppose $G$ is residually finite. Then for
  every element $x\in G$ the centralizer
  $$C_G(x)=\{g\in G\mid gx=xg\}$$
  is separable.
\end{lemma}

\begin{proof}
  Let $g\in G\smallsetminus C_G(x)$, so $gxg^{-1}x^{-1}\neq 1$. We
  need to find a finite index subgroup $G'<G$ such that $g\not\in G'$
  but $G'\supset C_G(x)$. By residual finiteness
  there is a finite quotient $\overline G$ of $G$ such that the above
  commutator maps nontrivially, i.e. the images $\overline x,
  \overline g$ of $x,g$ do not commute. Then let $G'$ be the preimage
  of $C_{\overline G}(\overline x)$. 
\end{proof}

We also note the following: {\it if $H<G$ is separable and the double
  coset $HgH$ is distinct from $H$ (i.e. $g\not\in H$) then there is a
  finite index subgroup $G'<G$ disjoint from $HgH$.} Indeed, take $G'$
so that $g\not\in G'\supset H$.

\subsection{Induction}\label{induction}

We observe:

{\it  If $H<G$ has finite index, and $H$ satisfies (QT), then so does $G$.}

 More generally, if
$H$ acts by isometries on a metric space $X$ with orbit maps $H\to X$
QI embeddings, then $G$ isometrically acts on $X^{[G:H]}$ 
with QI orbit maps. This is seen by the standard induction
construction. Define
$$Y=Map_H(G,X)$$
as the set of $H$-equivariant functions $G\to X$ where $H$ acts on $G$
by left multiplication (and on $X$ on the left). Then $Y$ is a
$G$-set via
$$(g\circ f)(\gamma)=f(\gamma g^{-1}).$$
Finally, as a metric space, $Y$ is isometric to
$X^{[G:H]}$: choose coset representatives $g_i$ so that $G=\coprod
Hg_i$ and define an isometry $Y\to X^{[G:H]}$ via
$$f\mapsto (f(g_1),f(g_2),\cdots,f(g_{[G:H]})).$$

\subsection{Projection complexes}\label{BBF}
In this section we review the construction of projection complexes in
\cite{bbf} with improvements from \cite{bbfs}.

\newcommand{\X}{{\mathcal X}}
\newcommand{\Z}{{\mathcal Z}}

The input is a collection $\bY$ of geodesic metric spaces and
for $X,Z\in\bY$ with $X\neq Z$ there is a projection $\pi_Z(X)\subset Z$. 
We also define
$d_Y(X,Z)=\diam \pi_Y(X)\cup\pi_Y(Z)$ for $X,Y,Z\in\bY$.
The pair $(\bY, \{\pi_Y\})$ satisfies the
 {\em projection axioms} for a {\em projection constant} $\xi\geq 0$ if
  \begin{enumerate}
  \item [(P0)] $\diam \pi_Y(X)\leq\xi$ when $X\neq Y$, (Bounded projection)
    \item [(P1)] if $X,Y,Z$ are distinct and $d_Y(X,Z)>\xi$ then
      $d_X(Y,Z)\leq\xi$, (Behrstock inequality)
    \item [(P2)] for $X\neq Z$ the set
      $$\{Y\in\bY\mid d_Y(X,Z)>\xi\}$$
      is finite. (Finiteness)
  \end{enumerate}

If we replace (P1) with
\begin{enumerate}
\item [(P1)$'$] if $X,Y,Z$ are distinct and $d'_Y(X,Z)>\xi$ then
      $\pi'_X(Y)=\pi'_X(Z)$
\end{enumerate}
then the collection satisfies the {\em strong projection axioms}. While there are many natural situations where the projection axioms hold, the strong projection axioms are not as natural. However, we can modify the projections so that they do hold. The following is proved in \cite[Theorem 4.1 and Lemma 4.13]{bbfs}.
\begin{thm}\label{strong axioms}
If the collection $(\bY, \{\pi_Y\})$ satisfies the projection axioms then there are projections $\{\pi'_Y\}$ such that $(\bY, \{\pi'_Y\})$ satisfy the strong projections axioms with projection constant $\xi'$ where $\pi'_Y(X)$ and $\pi_Y(X)$ are a uniform Hausdorff distance apart and $\xi'$ only depends on $\xi$.
\end{thm}

Let
$\cC_K(\bY)$ denote the space obtained from the disjoint union
$$\coprod_{Y\in\bY} Y$$
by joining points in $\pi_X(Z)$ with points in $\pi_Z(X)$ by an edge of length one whenever $d_Y(X,Z) <K$  for all $Y\in\bY\backslash\{X,Z\}$. When the spaces are
graphs and projections are subgraphs we can join just the vertices in
these projections. If a group $G$ acts on the disjoint union of $Y\in\bY$ by isometries and the $\pi_Y$ are $G$-invariant, ie, $\pi_{gY}(gX)=g\pi_Y(X)$, then $G$ acts isometrically on $\cC_K(\bY)$. 

\begin{thm}[\cite{bbf}]\label{bbf}
If $(\bY, \{\pi_Y\})$ satisfy the strong projection axioms with projection constant $\xi$ then for all $K>2\xi$ 
\begin{itemize}
\item   $\cC_K(\bY)$ is hyperbolic if all $Y\in\bY$ are $\delta$-hyperbolic;
\item $\cC_K(\bY)$ is a quasi-tree if all $Y\in\bY$ are quasi-trees with uniform
  QI constants.
  \end{itemize}
\end{thm}

There is a very useful distance formula in $\cC_K(\bY)$, see \cite[Theorem
  6.3]{bbfs}. Let $X,Z\in\bY$ and $x\in X$, $z\in Z$. We define
$d_Y(x,z)=d_Y(X,Z)$ if $Y\neq X,Z$, $d_X(x,z)=\diam
(\{x\}\cup\pi_X(Z))$ if $X\neq Z$, and $d_X(x,z)$ is the given
distance in $X$ if $X=Z$. We also define the {\em distance function with threshold $K$} by 
$$d_Y(,)_K =  \left\{\begin{array}{cl} d_Y(,) &\mbox{ if }d_Y(,) \ge K\\  0 & \mbox{ otherwise.}\end{array} \right.$$

\begin{prop}\label{distfor2}
 Let $(\bY, \{\pi_Y\})$ satisfy the strong projection axioms with projection constant $\xi$.
  Let $x\in X$ and $z\in Z$ be two points of $\cC(\bY)$ with
  $X,Z\in\bY$. Then
  $$\frac 14 \sum_{Y\in\bY} d_Y(x,z)_K\leq
  d_{\cC_K(\bY)}(x,z)\leq 2\sum_{Y\in\bY}d_Y(x,z)_K+3K$$
  for all $K \ge 4\xi$.
\end{prop}

\section{Proof of Theorem \ref{thm1}}

For simplicity all metric spaces will be graphs with each edge of
length 1 (and subspaces will be subgraphs).

\subsection{Projection axioms in $\delta$-hyperbolic spaces}
\begin{prop}\label{hyperbolic axioms}
Let $\bY$ be a collection of quasi-convex subspaces (with uniform constants) in a $\delta$-hyperbolic space $\mathcal Y$. For $X,Y\in \bY$ let $\pi_Y(X)$ be the nearest point projection and assume that $\diam \pi_Y(X) \le \theta$ (for all $X\neq Y\in\bY$). Then $(\bY, \{\pi_Y\})$ satisfies the projection axioms with projection constant $\xi$.
\end{prop}

\begin{proof}
Axiom (P0) holds with constant $\theta$ by assumption. Given $X$ and
$Z$ in $\bY$ let $\gamma$ be a shortest geodesic from $X$ to $Z$. Then
for any other $Y\in \bY$ the nearest point projection of $Y$ to
$\gamma$ will have diameter uniformly close to $d_Y(X,Z)$. To see this
let $\alpha$ and $\beta$ be shortest paths between $X$ and $Y$ and
between $Y$ and $Z$, respectively. The right endpoint of $\alpha$ will
lie in $\pi_Y(X)$ and the left endpoint of $\beta$ will lie in
$\pi_Y(Z)$. Let $\gamma'$ be a shortest path connecting these
endpoints. If $\gamma'$ is sufficiently long, when we concatenate
these three geodesics we get a quasi-geodesic with coarsely the same
endpoints as $\gamma$ so it will fellow travel $\gamma$. By
construction the nearest point projection of $Y$ to the quasi-geodesic
will be coarsely $\gamma'$ (whose diameter is roughly $d_Y(X,Z)$) so
the nearest point projection of $Y$ to $\gamma$ will also roughly be
$d_Y(X,Z)$.

This directly implies (P1) since the distance from $Z$ to $\alpha$ will be coarsely bounded below by $d_Y(X,Z)$ and hence, if this term is large, the strongly contracting property of $\delta$-hyperbolic space implies that the projection of $Z$ to $\alpha$ has uniformly bounded diameter. By the above assertion this diameter is coarsely $d_Z(X,Y)$ so this quantity is  also bounded. 

For (P2) if $X, Z, Y_1, \dots,Y_n$ are all in $\bY$ with $d_{Y_i}(X,Z)$ large then the sum of the projections is bounded by, say, twice the distance between $X$ and $Z$. If not there would be a $Y_i$ and $Y_j$ whose projection to $\gamma$ have large diameter overlap which would imply that $\diam \pi_{Y_i}(Y_j)>\theta$.
\end{proof}

\noindent
{\bf Quasi-geodesics.} A quasi-geodesic is a subspace of a metric space that is quasi-isometric to $\ZZ$. For our purposes it will be convenient to assume that quasi-geodesics are a collection of bi-infinite paths parameterized by arc length. In the proof of Theorem \ref{thm1} our quasi-geodesics will be a single bi-infinite path. However, when we discuss mapping class groups we will need quasi-geodesics that are finite union of paths.

\bigskip

We now prove a sequence of technical results that will be needed in what follows. 

We begin with a general setup:
\begin{itemize}
\item $\mathcal Y$ is a $\delta$-hyperbolic geodesic metric space. For convenience will assume that $\mathcal Y$ is a metric graph with edges of length one.

\item $\tilde\A$ is a collection of quasi-geodesics in $\mathcal Y$ with uniform constants.

\item $\A\subset \tilde\A$ is a sub-collection.

\item For each distinct $\alpha, \beta\in\A$ the projection $\pi_\alpha(\beta)$ is a subset of $\alpha$ that is uniformly close (in the Hausdorff metric) to the nearest point projection.

\end{itemize}
 We will refer to $\delta$, the quasi-geodesic constants and the Hausdorff bound on the distance of the projections from the nearest point projections as the {\em coarse constants}.

The following is a direct consequence of Proposition \ref{hyperbolic axioms}.
\begin{thm}\label{hyperbolic projections}
Fix $\theta$. Then there exists $\xi$, depending only on the coarse constants and $\theta$, such that if $\diam \pi_\alpha(\beta) \le \theta$ for all distinct $\alpha$ and $\beta$ in $\A$ then $(\A, \{\pi_\gamma\})$ satisfies the projection axioms with projection constant $\xi$.
\end{thm}

The following proposition is the main estimate we need to approximate lengths in $\mathcal Y$ using our quasi-trees.

\begin{prop}\label{main estimate}
Fix constants $R,K>0$.
Then there exists an $L>0$, depending only on the coarse constants and $K$, such that the following holds. Assume that
\begin{itemize}
\item any path of length $L$ in some $\alpha\in \tilde\A$ is contained in some $\gamma\in\A$;

\item $\A$ is partitioned into $\A_1\sqcup\dots\sqcup\A_n$;

\item for all $x,y\in\mathcal Y$  there is an $\alpha \in \tilde\A$ that intersects the $R$-neighborhood of both $x$ and $y$;

\item $\hat x =\{x_1, \dots, x_n\}$ and $\hat y=\{y_1,\dots, y_n\}$ are $n$-tuples of vertices in $\mathcal Y$ that are contained in the $R$-neighborhoods of $x$ and $y$, respectively.

\end{itemize}
Then
$$d_{\mathcal Y}(x,y)  \le 2\sum_i\sum_{\gamma\in\A_i} d_\gamma(x_i,y_i)_K +L+2R.$$
\end{prop}

\begin{proof}
We can assume that $d_\mathcal Y(x,y) \ge2R$. 
Choose an $\alpha \in \tilde\A$ that intersects the $R$-neighborhood of both $x$ and $y$. Let $\tilde\alpha$ be the subpath of $\alpha$ between $x$ and $y$ that is disjoint  from the $R$-neighborhoods of $x$ and $y$ but whose endpoints are exactly $R$ from $x$ and $y$. As geodesics (and hence quasi-geodesics) are strongly contracting in a $\delta$-hyperbolic space, the projection of the $R$-neighborhood of $x$ to any subpath of  $\tilde\alpha$ will be contained in a uniformly bounded neighborhood of the endpoint of the path closest to $x$ (with the equivalent statement holding for the $R$-neighborhood of $y$).
Therefore, for all $\beta\in\A$ that intersect $\tilde\alpha$ and all $x',y'\in\mathcal Y$ with $d_\mathcal Y(x,x'), d_\mathcal Y(y,y') \le R$, we have that
$$\diam(\tilde\alpha\cap\beta)-d_\beta(x',y')$$
will be bounded above by a constant that only depends only on the coarse constants (and not  on $R$). Using that quasi-geodesics in $\tilde\A$ have uniform constants, if $\tilde\alpha\cap\beta$ contains a subpath of sufficient path length then $\diam(\tilde\alpha\cap\beta)$ will be large. When the diameter is large we can absorb the above additive error into a multiplicative one. Therefore there exists an $L>0$ such that if $\tilde\alpha\cap\beta$ contains a path of length $L$ then
\begin{itemize}
\item $\diam(\tilde\alpha\cap\beta) \le 2d_\beta(x',y')$ and
\item $\diam(\tilde\alpha\cap\beta) \ge 2K.$
\end{itemize}
Combining these estimates we have
$$\diam(\tilde\alpha\cap\beta) \le 2d_\beta(x,y)_K$$
if $\tilde\alpha\cap\beta$ contains a path of length $L$.

If $d_{\mathcal Y}(x,y)\ge L+2R$ then $\tilde\alpha$ will be a path of length at least $L$ and by the choice of $\A$ we can find distinct axis $\gamma_1,\dots, \gamma_m$ in $\A$ such that each $\gamma_i\cap \tilde\alpha$ contains a segment of length $L$ and the union of the intersections is all of $\tilde\alpha$. We let $\A_{j_i}$ be the subcollection in the partition of $\A$ that contains $\gamma_i$. Using the above estimate we then have
\begin{eqnarray*}
d_{\mathcal Y}(x,y)&\le& \sum_i \diam(\tilde\alpha\cap\gamma_i) + 2R\\
&\le & 2 \sum_i d_{\gamma_i}(x_{j_i},y_{j_i})_K +2R\\
&\le& 2\sum_i\sum_{\gamma\in\A_i} d_\gamma(x_i,y_i)_K +2R.
\end{eqnarray*}
If $d_\mathcal Y(x,y) < L +2R$ then the sum in the inequality may be zero. However, if we add $L$ to the right then the inequality will still hold in this case completing the estimate.
\end{proof}

We now assume that $G$ acts isometrically on $\mathcal Y$ and that $\tilde\A$ is $G$-invariant. The action is {\em acylindrical} if for any $\epsilon>0$ there exists $D,B>0$ such that if $x,y, x', y'\in\mathcal Y$ with $d_\mathcal Y(x,y)>D$ then the set
$$\{g\in G| d_\mathcal Y(x',gx), d_\mathcal Y(y', gy) \le \epsilon\}$$
has at most $B$ elements. This is slightly different than the usual definition where one assumes that $x=x'$ and $y=y'$. It is not hard to check that the two definitions are equivalent.

In the proof of Theorem \ref{thm2} we will consider the action of the
mapping class group on the curve graphs of essential subsurfaces. For
this reason we will need to consider actions where there is  a large kernel. In particular assume that $\tilde G$ acts on $\mathcal Y$  and $G$ is the image of $\tilde G$ in the isometry group of $\mathcal Y$. If the kernel of the quotient map $\tilde G\to G$ is infinite then the action of $\tilde G$ cannot be acylindrical. However, $G$ may act acylindrically on $\mathcal Y$ in which case we say that the action of $\tilde G$ has {\em acylindrical image}.

Let $\gamma \in \tilde\A$ be the axis of an element $g$ that acts
hyperbolically on $\mathcal Y$.

We
let $\tilde C(\gamma)$ be the subgroup of $\tilde G$ that fixes
$\gamma$, up to bounded Hausdorff distance, and $\A_\gamma$ the 
equivalence classes, with respect to bounded Hausdorff distance,
of the $\tilde G$-translates of $\gamma$. 
There
is a natural bijection between the set of left cosets of $\tilde
C(\gamma)$ and $\A_\gamma$. The group $\tilde G$ acts on
$\A_\gamma$.
Accordingly, we need to define (the diameter of) the projection between the equivalence
classes: 
$$\diam \pi_{[\gamma]}([\beta])= \sup_{\gamma' \in [\gamma],
\beta' \in [\beta]} \diam \pi_{\gamma'}(\beta').$$
Since the quasi-geodesic constants are uniformly bounded, 
the difference between $\diam \pi_{[\gamma]}([\beta])$
and $\diam \pi_{\gamma}(\beta)$ is uniformly bounded. 

It will be convenient to assume that
if $h\in \tilde G$ and $\gamma$ and $h(\gamma)$ are a bounded
Hausdorff distance from each other then $h(\gamma) = \gamma$.
So, we will work with this assumption, rather than 
work with the equivalence classes and the modified 
projection, which could be done with bounded modification 
of the constants in the argument. 

 If $g\in\tilde C(\gamma)$ then, as $g$ fixes $\gamma$ (under our assumption),
for any $\beta\in\A_\gamma$ we have that $\diam \pi_\gamma(\beta) =
\diam \pi_\gamma(g(\beta))$. In particular, any two axes that are
translates of $\gamma$ by elements in
the same double coset of $\tilde C(\gamma)$ have projections to
$\gamma$ with the same diameter.

\begin{prop}\label{acylind_finite}
If $\tilde G$ acts on $\mathcal Y$ with acylindrical image then there exists a $\theta>0$, depending only on the coarse constants and the acylindrical constants, such that only finitely many double cosets of $\tilde C(\gamma)$ have projection to $\gamma$ of diameter $>\theta$.
\end{prop}

\begin{proof}
First we can replace $\tilde G$ with its image $G$ in the isometry group of $\mathcal Y$. This is because the subgroup $C(\gamma)$ of $G$ that fixes $\gamma$ will be the image of $\tilde C(\gamma)$ under the quotient map $\tilde G\to G$ and the kernel of this quotient map will be also be the kernel of the quotient map $\tilde C(\gamma) \to C(\gamma)$. Therefore the quotient map $\tilde G\to G$ induces a bijection between the double cosets of $\tilde C(\gamma)$ in $\tilde G$ and of $C(\gamma)$ in $G$.

There is an $\epsilon>0$, only depending on the coarse constants, such that if $\alpha, \beta \in \A$ then the difference between the diameter of $\pi_\alpha(\beta)$ and the diameter of the intersection of $\beta$ with the $\epsilon$-neighborhood of $\alpha$ is uniformly bounded. Let $D>>\epsilon$ be the acylindricity constant for $\epsilon$. Let $\tilde\gamma$ be a finite subpath whose diameter is at least $4D$ and that contains at least two copies of a fundamental domain for the $C(\gamma)$ action on $\gamma$.

Assume that for $g\in C(\gamma) h C(\gamma)$ the translate $g(\gamma)$
has large projection to $\gamma$ where ``large'' roughly means at
least $2D$.
Then we can assume that the coset representative $h$ has been chosen such that there is subpath $\gamma_h$ of $\tilde\gamma$ such that $h(\gamma_h)$ is contained in the $\epsilon$-neighborhood of $\tilde\gamma$ and $\diam \gamma_h \ge 2D$. This implies that the endpoints of $h(\gamma_h)$ will be contained in the $\epsilon$-neighborhood of two vertices $x_h$ and $y_h$ of $\tilde\gamma$ with $d_\mathcal Y(x_h, y_h) \ge D$. Note that there are finitely many triples $\{x,y,\alpha\}$ where $x,y\in \tilde\gamma$ with $d_\mathcal Y(x,y) \ge D$ and $\alpha$ is a subpath of $\tilde\gamma$ of diameter $\ge 2D$. By acylindricity for each triple $\{x,y,\alpha\}$ there are finitely many $h$ such that $x_h = x$, $y_h = y$ and $\gamma_h = \alpha$. This implies that there are finitely many double cosets with projection roughly larger than $2D$.
\end{proof}

\subsection{Axes}\label{2.2}
By the induction in Section \ref{induction} we may replace $G$ by a
finite index subgroup.
Thus by residual finiteness we may assume $G$ is torsion free (recall that
hyperbolic groups contain finitely many conjugacy classes of torsion
elements). In particular, if $\<g\>$ is a maximal cyclic subgroup,
then the centralizer (and also normalizer) of $g$ is $\<g\>$
itself, which is therefore separable by Lemma \ref{separability}.

The following is surely well known. We summarize the proof.
\begin{thm}\label{quasi-axes}
Let $G$ be a torsion free $\delta$-hyperbolic group and $\Gamma(G)$ a Cayley graph for some finite generating set. Then there exists a $G$-invariant collection $\tilde\A$ of axes of maximal cyclic subgroups where the axes are uniform quasi-geodesics. Furthermore any $x$ and $y$ are within uniform distance $R$ of an axis in $\tilde\A$.
\end{thm}

\begin{proof}
 Let $|g|$ be the word norm with respect to the chosen generating set.
 In each conjugacy class of maximal
cyclic subgroups choose a representative $\<g\>$ with $|g|$ minimal
possible. Define the axis $\gamma_g$ as the union of $g$-translates of a geodesic segment from $1$ to $g$ and we assume that $\gamma_g = \gamma_{g^{-1}}$. We then extend the
definition to the conjugates: $\gamma_{aga^{-1}}=a\gamma_g$ (this is
well-defined by the remark about normalizers). These are
axes of indivisible elements; each $g$ acts by translation on its
axis. 
Let $\tilde\A$ be the collection of all such axes.
Moreover, it is a well-known fact that each axis is a quasi-geodesic with uniform
constants. Indeed, if $|g|$ is large compared to $\delta$, 
say, $|g| > 1000 \delta$ then there are uniform constants depending only on $\delta$ by \cite[7.2C]{G}.
Now, there are only finitely many elements $g$ with $|g| \le 1000 \delta$, 
so the claim follows. 
It is also standard that there exists a constant $R$ such
that for any two elements $x,y\in G$ there exists an axis $\gamma_g$
that intersects the $R$-balls centered at $x$ and $y$.
This is a consequence of the fact \cite[8.2G]{G} that the set of pairs
$(\gamma_g^{\infty}, \gamma_g^{-\infty}) \in \partial \Gamma(G) \times \partial \Gamma(G)$
for all $g$ of infinite order is dense in $\partial \Gamma(G) \times \partial \Gamma(G)$.
\end{proof}
\subsection{Constants}
We can now fix constants. The action of a group on its Cayley graph is proper and therefore acylindrical. By Proposition \ref{acylind_finite} there is a $\theta>0$ such that for any axis $\gamma\in \tilde \A$ there are only finitely many double cosets of $C(\gamma)$ that have projection to $\gamma$ with diameter $> \theta$. By Theorem \ref{hyperbolic projections}, there exist a $\xi'>0$ such that for any subcollection of $\tilde\A$ where the projections have diameter bounded by $\theta$, the subcollection satisfies the projection axioms with projection constant $\xi'$.
By Theorem \ref{strong axioms} the projections can be modified to satisfy the strong projection axioms with projection constant $\xi$ only depending $\xi'$. We then let $K=4\xi$ so that the distance formula, Proposition \ref{distfor2}, holds with threshold $K$. We then fix the segment constant $L$ to satisfy Proposition \ref{main estimate} for $K$.
\subsection{Preferred axes}\label{s:axes}
We now choose the $G$-finite and $G$-invariant collection of preferred axes $\A$
(or equivalently, conjugacy classes of indivisible elements).
We view the axes in $\tilde\A$ as a collection of bi-infinite words in the generators and for every word $x$ of length $\leq L$
choose, if possible, an element $\gamma_x\in \tilde\A$ such that $x$ is a subword of $\gamma_x$. Then
let $\A$ be the collection of $G$-orbits of the selected axes. Note that every such $x$ will not necessarily be a subword for an axis in $\A$ even if $x$ is a geodesic but every subword $x$ of length $\le L$ in an axis in $\gamma \in \tilde \A$ will be contained in an axis $\beta \in \A$ with $x \subset \gamma\cap \beta$.

\subsection{Coloring $\A$}\label{2.5}
Let $\gamma_1, \dots, \gamma_n$ represent the distinct $G$-orbits of axes in $\A$.  Then for each $\gamma_i$, $C(\gamma_i)$ is an infinite cyclic group and is its own centralizer. As $G$ is residually finite, by Lemma \ref{separability} the subgroup $C(\gamma_i)$ is separable. Given $h\not\in C(\gamma_i)$ there is a finite index subgroup of $G$ that contains $C(\gamma_i)$ but not $h$ and therefore doesn't contain the double coset $C(\gamma_i)h C(\gamma_i)$. Using Proposition \ref{acylind_finite} we can therefore find a finite index subgroup $H_i$ such that the projection between any two axes in the $H_i$-orbit of $\gamma_i$ (or the $H_i$-orbit of any axis in the $G$-orbit of $\gamma_i$) have diameter $\le \theta$.

Let
$$H' = H_{1}\cap \dots \cap H_{n}$$
and let $H$ be the intersection of the $G$-conjugates of $H'$.
Now add axes $\gamma_{n+1}, \dots, \gamma_m$ so that we have one axis in $\A$ for each $H$-orbit.
Let $\A_i$ be the $H$-orbit of $\gamma_i$. We then have:
\begin{cor}
There is a finite index subgroup $H$ of $G$ and a partition $\A_1 \sqcup \dots \sqcup \A_m$ of $\A$ such that each $\A_i$ is $H$-invariant and the projections between any two axes in a fixed $\A_i$ have diameter $\le \theta$.
\end{cor}

\subsection{Product of quasi-trees $\X$}
By Theorem \ref{bbf}  and Proposition \ref{distfor2} for each $\A_i$ we have a quasi-tree $\C_{K}(\A_i)$ that has an isometric $H$-action and a lower bound on distance
$$\frac14\sum_{\gamma\in\A_i} d_\gamma(x,y)_{K} \le d_{\C_{K}(\A_i)}(x,y)$$
where $x$ and $y$ lie on axes in $\A_i$.

Let
$$\X= \prod_{i=1}^m \C_{K}(\A_i)$$
be the product of quasi-trees. We give $\X$ the $\ell^1$-metric (which is quasi-isometric
to the $\ell^2$-metric). If $\hat x$ and $\hat y$ are $m$-tuples representing elements in $\X$ with the $i$th coordinate lying in axis in $\A_i$ then we sum the distance bound to get
$$\frac14\sum_i\sum_{\gamma\in\A_i} d_\gamma(x_i,y_i)_{K} \le d_{\X}(\hat x,\hat y) $$

Fix $\hat x$ as a basepoint. We claim that the orbit map $H\to \X$ given by $h\mapsto h(\hat x)$ is a quasi-isometric embedding. As $H$ is finite index in $G$ it is quasi-isometrically embedded in $\Gamma(G)$ so we need to show that $d_\X(\hat x, h(\hat x))$ is bounded above and below by linear functions of the word length $|h|$. (We emphasize that the word length is for the generators of $G$ we chose in Theorem \ref{quasi-axes}.) The upper bound is clear since orbit maps are Lipschitz. The union $\hat x \cup \{id\}$ is a finite set and therefore has diameter in $\Gamma(G)$ bounded by some $R>0$. By Theorem \ref{quasi-axes}, after possibly enlarging  $R$, we can also assume that for all $h \in H$ there is an axis in $\tilde A$ that intersects the $R$-neighborhoods in $\Gamma(G)$ of both $id$ and $h$.
By Proposition \ref{main estimate} we have
$$|h| \le 2\sum_i\sum_{\gamma\in\A_i} d_\gamma(x_i, h(x_i))_K +L+2R$$
and therefore
$$\frac 18\left(|h| - L -2R\right) \le d_\mathcal X(\hat x, h(\hat x)).$$
This completes the proof of Theorem \ref{thm1}.

\section{Proof of Theorem \ref{thm2}}
Let $\Sigma$ be a closed surface with finitely many marked points and
let $MCG(\Sigma)$ be the mapping class group of $\Sigma$.  The rest of
the paper is devoted to the proof of Theorem \ref{thm2}, that
$MCG(\Sigma)$ embeds in a product of quasi-trees. The general outline
closely follows our proof of Theorem \ref{thm1}, but there are several
complications that arise. The central one is that $MCG(\Sigma)$ is not
a hyperbolic group. However, by the Masur-Minsky distance formula it
does embed in an infinite product of hyperbolic spaces, the curve
graphs for subsurfaces of $\Sigma$.  In \cite{bbf}, we used projection
complexes to embed $MCG(\Sigma)$ in a finite product of hyperbolic
spaces where the Masur-Minsky distance formula was a key
ingredient. We would like to follow the strategy of the proof of
Theorem \ref{thm1} to embed each curve graph in finite product of
quasi-trees. However, curve graphs are locally infinite so this adds a
new difficulty.

To see this difficulty let us focus on the main factor, the
curve graph $\C(\Sigma)$ of the surface $\Sigma$.
 If we mimic the
construction in Section \ref{s:axes}, we would take axes of all pseudo-Anosov
elements of some bounded translation length. However, this would give us infinitely
many conjugacy classes and the coloring construction in Section
\ref{2.5} will break down. To fix this problem we restrict to a finite collection of conjugacy classes that contain every {\em thick} segment of bounded length. This amounts to requiring the axes in Teichm\"uller space fellow travel every geodesic segment in a fixed thick part, but we will develop this notion combinatorially, in terms of Masur-Minsky subsurface projections. This will give an embedding of the thick part of the curve graph in a finite product of quasi-trees but not a quasi-isometric embedding of the entire curve graph. The distance lost will be picked up in curve graphs of proper subsurfaces. This is captured more formally in our {\it thick distance formula}, a version
of the Masur-Minsky distance formula that counts only long segments
that are thick in some subsurface. With these modifications, the proof of Theorem \ref{thm1}
will generalize to mapping class groups.

\subsection{Curve graphs and subsurface projections}
We set some notation. The {\em curve graph} of $\Sigma$ is denoted $\C(\Sigma)$. If $Y$ is an essential (connected, compact and possibly punctured) subsurface that is not a triply punctured sphere then its curve graph is also denoted $\C(Y)$. If $x$ is a curve in $\C(\Sigma)$ then $x$ is disjoint from $Y$ if it can be homotoped in $\Sigma$ to be disjoint from $Y$. Otherwise $x$ {\em cuts} $Y$. If $x$ cuts $Y$ we let $\pi_Y(x)$ be the projection of $x$ to $Y$. If $\tilde x$ is a collection of curves then $\pi_Y(\tilde x)$ is the union of $\pi_Y(x)$ for those $x\in\tilde x$ that cut $Y$. If some component of $\partial X$ cuts $Y$ then $\pi_Y(X) = \pi_Y(\partial X)$. Two subsurfaces $X$ and $Y$ are {\em transverse} if a component of  $\partial X$ cuts $Y$ and a component of $\partial Y$ cuts $X$. We refer to \cite{MM}, \cite{MM2} for precise definitions.

The next result plays a central role in the paper.
\begin{thm}\label{axioms-hold}
There exists a universal constant $\xi>0$ such that the following holds.
\begin{itemize}
\item If $\bY$ is a collection of pairwise transverse subsurfaces then $(\bY, \{\pi_Y\})$ satisfy the projection axioms with projection constant $\xi$
(see \cite[Section 5]{bbf}).
\item If $x,z\in \cC(\Sigma)$ and $d_Y(x,z)> \xi$ then every geodesic in $\cC(\Sigma)$ from $x$ to $z$ contains a curve disjoint from $Y$ (\cite[Theorem 3.1]{MM2}).
\end{itemize}
\end{thm}
These two results are usually stated separately but it will be convenient for us to have the same constant for both. The second bullet is the {\em Bounded Geodesic Image Theorem} and we will reference it below as BGIT. Sometimes the contrapositive will also be useful: If every curve in the geodesic cuts $Y$ then the projection of the geodesic to $Y$ has diameter $< \xi$ in $\cC(Y)$.

\subsection{The Masur-Minsky distance formula}
Recall the Masur-Minsky distance formula for word length in the mapping class group
(\cite[Theorem 6.12]{MM2}, cf. \cite[Section 2]{bbf} for this form).
\begin{thm}[Masur-Minsky distance formula]\label{MM.formula}
Let $\tilde x$ be a collection of filling curves on $\Sigma$. Then for $R$ sufficiently large 
the word length $|g|$ (with respect to some fixed generating set) is bounded above and below by linear functions of
$$\sum_{Y\subseteq \Sigma} d_Y(\tilde x, g(\tilde x))_R.$$
\end{thm}
A collection of curves $\tilde x$ is {\em filling} if every curve in $\C(\Sigma)$ intersects some curve in $\tilde x$.

We will need a new version of this distance formula where length is only measured in the {\em thick} part of the curve graph. We need some more setup before we state the formula.

\subsubsection{Bounded pairs and finiteness}

The curve graph is not locally finite. The following concept is the
replacement for this lack of local finiteness. See also \cite{kasra-saul}.
  
\begin{definition}[$T$-thick]
A collection of curves $\tilde x$ is $T$-thick if for all $x,z\in \tilde x$ and all proper subsurface $Y$ we have $d_Y(x,z) \le T$.
\end{definition}

\begin{thm}[\cite{choi-rafi,wat-fin}]\label{yohsuke}
Given any $C>0$ there exists a $D>0$ such that if $x$ and $y$ are in
$\C(\Sigma)$ and $i(x,y) \ge D$ where $i(x,y)$ is the geometric
intersection number, then $d_{\C(Y)}(x,y)\geq C$ for some subsurface
$Y\subseteq \Sigma$.
\end{thm}

Up to the action of the mapping class group there are only finitely many curves of bounded intersection. This gives the following corollary.
\begin{cor}\label{finite-thick}
Up to the action of the mapping class group there are finitely many collections of $T$-thick curves in $\C(\Sigma)$ that have diameter in $\C(\Sigma)$ bounded by $T$.
\end{cor}

\subsubsection{Tight geodesics}

If both $x,z \in \C(\Sigma)$ cut $Y$, then we define
$$d_Y(x,z) = \diam_{\C(Y)}(\pi_Y(x) \cup \pi_Y(z)).$$
If $g$ is a geodesic in $\C(\Sigma)$ connecting $x$ to $z$  and $x'$ and $z'$ are endpoints of a subsegment then there is no general relationship between $d_Y(x,z)$ and $d_Y(x',z')$. However, if we restrict to the special class of tight geodesics then we will get bounds. We say that a geodesic $g=\{x_0, x_1, \dots, x_n\}$ is {\em tight} if $x_i$ is a component of the boundary of the surface filled by $x_{i-1}$ and $x_{i+1}$ for $0<i<n$. By \cite[Lemma 4.5]{MM2} there is a tight geodesic connecting any two curves in $\C(\Sigma)$.

\begin{lemma}\label{tight-bound}
Assume that $x'$ and $z'$ lie on a tight geodesic between $x$ and $z$ and that $x'$ and $z'$  cut a subsurface $Y$. If $d_Y(x',z') \ge \xi$ then $x$ and $z$ cut $Y$ and 
$$d_Y(x',z') < d_Y(x,z) + 2\xi.$$ In particular if $x$ and $z$ are
$T$-thick then the collection of curves in a tight geodesic from $x$
to $z$ is $(T+2\xi)$-thick.
\end{lemma}

\begin{proof}
Since $d_Y(x',z') \ge \xi$ by the BGIT there is a $y'$ in between $x'$
and $z'$ that is disjoint from $Y$. If there is another $y \in g$ that
is disjoint from $Y$ and is in the complement of the segment between
$x'$ and $y'$ then $d_{\C(\Sigma)}(y,y') \le 2$ and therefore we must
have that $y$ and $y'$ are both adjacent to either $x'$ or $y'$. By
tightness a curve that is adjacent to two curves that are disjoint
from $Y$ will also be disjoint from $Y$. This is a contradiction, so everything in the complement of the segment from $x'$ to $z'$ cuts $Y$.  In particular $x$ and $z$ cut $Y$.

If $d_Y(x,x') \ge \xi$ there is a $y$ between $x$ and $x'$ that is disjoint from $Y$, contradicting what we have just shown. Therefore $d_Y(x,x') < \xi$ and by the same argument $d_Y(z',z) < \xi$. By the triangle inequality
$$d_Y(x',z') \le d_Y(x,z) + d_Y(x,x') + d_Y(z,z') < d_Y(x,z) + 2\xi.$$
\end{proof}

\noindent
{\bf Convention.}
 Given a constant $T$ let $\hat T = T +2\xi$ and $\check T = T-2\xi$.

\bigskip

\subsubsection{Thick distance}
Given filling collections $\tilde x,\tilde z$ on $\Sigma$ and a subsurface $Y \subset \Sigma$,
we define
$$\Omega_T(Y;\tilde x,\tilde z) = \{Z \subseteq Y | Z \not= Y, d_Z(\tilde x,\tilde z)>T\}.$$
This set has an order coming from inclusion. Let $\Omega^m_T(Y;\tilde x,\tilde z)$
be the subset of maximal elements.

\begin{lemma}\label{at_most_two}
Given filling collections  $\tilde x,\tilde z$ on $\Sigma$, 
and a subsurface $Z\subset \Sigma$ there are at most two subsurfaces $Y$ with $d_Y(\tilde x,\tilde z) > T$ such that $Z \in \Omega^m_T(Y; \tilde x,\tilde z)$. 
\end{lemma}
\begin{proof}
Let $Y_0, Y_1, Y_2$ be subsurfaces such that $d_{Y_i}(\tilde x,\tilde z) > T$ and $Z \in \Omega^m_T(Y_i;\tilde x,\tilde z)$.
If $Y_i \subset Y_j$ for $i\neq j$ then $Y_i \in \Omega_T(Y_j; \tilde x,\tilde z)$ so $Z \not\in \Omega^m_T(Y_j; \tilde x,\tilde z)$, a contradiction. Therefore the $Y_i$ are mutually transverse.

Now choose an $x\in\pi_Y(\tilde x)$ and $z\in \pi_Y(\tilde z)$ that both cut $Z$ (and hence the $Y_i$).
By the ordering (see e.g. \cite[Theorem 3.3(G)]{bbf})  we have that
two of the subsurfaces have a large projection to the the third. We
can assume that is $Y_1$ and $|d_{Y_1}(x,z) - d_{Y_1}(Y_0,Y_2)| \le
\xi$. In particular $d_{Y_1}(Y_0,Y_2)$ is large, so that
 $\partial Y_0$ and $\partial Y_2$ fill $Y_1$ and that if
$Z \subset Y_1$ then it must intersect either $\partial Y_0$ or
$\partial Y_2$, a contradiction.
\end{proof}

We give a key definition. 
\begin{definition}[$T,R$-thick distance]
Fix sufficiently large constants $T,R$. 
Let $x_1,\dots,
x_n \in \cC(Y)$ be curves occurring in this order on a tight geodesic
in $\C(Y)$
from $x$ to $z$ such that $d_{\cC(Y)}(x_{2i-1}, x_{2i}) \ge R$ and
$d_Z(x_{2i-1}, x_{2i}) \le T$ for all $Z \subsetneq Y$.
Then the {\it $T,R$-thick distance}
$d_Y^{T,R}(x,z)$ is the maximum of $\sum d_{\cC(Y)}(x_{2i-1},x_{2i})$
over all such choices for the $x_i$, and for the tight geodesics from $x$ to
$z$. For collections of curves $\tilde x$ and $\tilde z$ in $\C(Y)$ we set $$d^{T,R}_Y(\tilde x, \tilde z) = \max_{x\in\tilde x, z\in\tilde z} d^{T,R}_Y(x,z).$$
If $\tilde x$ and $\tilde y$ are collections in $\C(\Sigma)$ we define
$$d^{T,R}_Y(\tilde x, \tilde z) = d^{T,R}_Y(\pi_Y(\tilde x), \pi_Y(\tilde z)).$$
\end{definition}

\begin{definition}[Footprint]
If $g$ is a geodesic in $\C(Y)$
and $Z\subset Y$ is a proper subsurface then the {\em footprint}
$F_Z(g)$ of $Z$ is the set of vertices of $g$ that are disjoint from
$Z$.  
Since any vertices of $\cC(Y)$
that are distance three or more apart will fill $Y$ 
the diameter of $F_Z(g)$ is at most two.  The footprint is connected for all $Z\subset \Sigma$ if and only if $g$ is a tight geodesic.
\end{definition}

If $Z' \subset Z$ then $F_{Z'}(g)\supseteq F_Z(g)$ but it may be that
strict inclusion holds. Whenever $F_Z(g)$ is nonempty we set
$F_Z^\subset(g)$ to be the union of $F_{Z'}(g)$ over all $Z' \subset
Z$. Note that if $y$ is in $F^\subset_Z(g)$ then $d_{\C(Y)}(y,z)\leq
2$ for any boundary component $z$ of $Z$. Thus the diameter of
$F^\subset_Z(g)$ will be at most four.

We state a key lemma. 

\begin{lemma}\label{thick_distance}

  For $T,R>0$ sufficiently large the following holds.
Let $\tilde x$ and $\tilde z$ be filling collections on $\Sigma$.
 Let $Y$ be a subsurface in $\Sigma$. Then,
$$d^{T,R}_Y(\tilde x,\tilde z) + (4+2R)|\Omega^m_{\check T}(Y; \tilde x,\tilde z)| \ge d_Y(\tilde x,\tilde z)_R$$
\end{lemma}

\begin{proof}
First  if $\Omega^m_{\check T}(Y;\tilde x,\tilde z)$ is empty, then $\Omega^m_{T}(Y;\tilde x,\tilde z)$
is empty and $d^{T,R}_Y(\tilde x, \tilde y) = d_Y(\tilde x, \tilde y)_R$. 
So assume $\Omega^m_{\check T}(Y;x,z)$ is not empty. 

Choose $x \in \pi_Y(\tilde x), z \in \pi_Y(\tilde z)$ and $g$ a tight geodesic between them such that $g$ realizes the thick distance $d_Y^{T,R}(\tilde x, \tilde z)$. Let $J$ be a subsegment of $g$ with endpoints $x'$ and $z'$. By Lemma \ref{tight-bound}
$$d_Z(x',z') < d_Z(x,z) + 2\xi$$
if $x'$ and $y'$ cut $Z$. In particular if $d_Z(x',z') \ge T$ then 
$$Z\in \Omega_{\check T}(Y;x,z) \subset \Omega_{\check T}(Y;\tilde x, \tilde z).$$
As every $Z\in \Omega_{\check T}(Y;\tilde x, \tilde z)$ is contained in some $Z' \in\Omega^m_{\check T}(Y;\tilde x, \tilde z)$ we have that if the interior of $J$ is disjoint from every $Z'\in \Omega^m_{\check T}(Y;\tilde x, \tilde z)$ then $d_Z(x',z') < T$ for all $Z\subset Y$.

 Let $J_0, \dots, J_n$ be a maximal collection of disjoint subsegments of $g$ such that the interiors of the $J_i$ do not contain any elements of $F^\subset_Z(g)$ for $Z\in\Omega^m_{\check T}(Y;\tilde x, \tilde z)$. Then the endpoints of any $J_i$ are $T$-thick.  As each $F^\subset_Z(g)$ is connected $n \le |\Omega^m_{\check T}(Y;x,z)|$.

Let 
$$\mathcal I = \{ i | 0\le i \le n \mbox{ and } |J_i| \ge R\}$$
and $\mathcal I'$ the complement. Then
\begin{eqnarray*}
\sum_{0\le i \le n} |J_i| = \sum_{i \in \mathcal I} |J_i| + \sum_{i\in \mathcal I'} |J_i| &\le& d_Y^{T,R}(x,z) + R(|\Omega^m_{\check T}(Y;x,z)| + 1)\\
& \le & d_Y^{T,R}(x,z) + 2R|\Omega^m_{\check T}(Y;x,z)|.
\end{eqnarray*}
As the diameter of each $F^m_Z(g)$ is bounded above by four, the length of the complement of the $J_i$ is bounded by $4|\Omega^m_{\check T}(Y;x,z)|$ so
\begin{eqnarray*}
d_Y(\tilde x, \tilde z) & = & d_{\C(Y)}(x, z) \\
& \le & \sum |J_i| + 4|\Omega^m_{\check T}(Y;x,z)| \\
& \le & d^{T,R}_Y(x,z) + (4+2R)|\Omega^m_{\check T}(Y;x,z)|
\end{eqnarray*}
and the lemma follows.
\end{proof}

By $cx(Y)$ denote the complexity of $Y$, i.e. the length of the
longest chain $Y=Y_0\supset Y_1\supset\cdots\supset Y_k$ of distinct
subsurfaces. 

\begin{thm}\label{thm:thick_distance}
Fix $T,R$ sufficiently large with $R \le \check T$.
Let $\tilde x,\tilde z$ be filling collections in $\C(\Sigma)$.
Then, for each $n$
$$\sum_{cx(Y) \le n} d_Y(\tilde x, \tilde z)_{\check T} \le \sum_{cx(Y) = n} d^{T,R}_Y(\tilde x, \tilde z) + (9+4R)\sum_{cx(Y)<n} d_Y(\tilde x, \tilde z)_{\check T}.$$
\end{thm}

We remark that each sum is over finitely many $Y$  since it is 
for $Y$ with $d_Y(\tilde x, \tilde z) \ge R$, and there are only finitely many 
such $Y$ for given $\tilde x,\tilde z$.

\begin{proof}
If $cx(Y) = n$ then by Lemma \ref{thick_distance},
\begin{eqnarray*}
d_Y(\tilde x, \tilde z)_{\check T} &\le& d_Y(\tilde x, \tilde z)_{R} \\
&\le& d^{T,R}_Y(\tilde x, \tilde z) + (4+2R)|\Omega^m_{\check T}(Y;\tilde x, \tilde z)| \\
& \le & d^{T,R}_Y(\tilde x, \tilde z) + (4+2R)\sum_{Z \in \Omega^m_{\check T}(Y;\tilde x, \tilde z)} d_Z(\tilde x, \tilde z)_{\check T}.
\end{eqnarray*}
By Lemma \ref{at_most_two}, any $Z$ will appear in at most two $\Omega_{\check T}^m(Y;\tilde x, \tilde z)$ and therefore if we sum the left hand side over all $Y$ with $cx(Y) = n$ we have
$$\sum_{cx(Y) =n} d_Y(\tilde x, \tilde z)_{\check T} \le \sum_{cx(Y)=n} d^{T,R}_Y(\tilde x, \tilde z) + (8+4R)\sum_{cx(Y)<n} d_Y(\tilde x, \tilde z)_{\check T}.$$
Adding $\sum_{cx(Y)<n} d_Y(\tilde x, \tilde z)_{\check T}$ to both sides gives the inequality.
\end{proof}

\begin{cor}\label{cor.thick_distance}
Let $\tilde x, \tilde z$ be filling collections in $\C(\Sigma)$.
Then
for sufficiently large $T,R$ with $R \le \check T$,
$$\sum_{Y\subset\Sigma} d_Y^{T,R}(\tilde x, \tilde z)
 \le 
\sum_{Y\subset\Sigma} d_Y(\tilde x, \tilde z)_{R} 
\le 
(9+4R)^{cx(\Sigma)-1} \sum_{Y\subset\Sigma} d_Y^{T,R}(\tilde x, \tilde z) 
 .$$
\end{cor}

\begin{proof}
The first inequality is trivial since 
$d_Y^{T,R}(\tilde x, \tilde z) \le d_Y(\tilde x, \tilde z)_{R} $ for all $Y$.
By inductively applying Theorem \ref{thm:thick_distance}, with base case $n=cx(\Sigma)$, we have
$$\sum_{cx(Y) \le cx(\Sigma)} d_Y(\tilde x, \tilde z)_{\check T} \le (9+4R)^{cx(\Sigma) - n}\left(\sum_{n \le cx(Y) \le cx(\Sigma)} d^{T,R}_Y(\tilde x, \tilde z) + (9+4R)\sum_{cx(Y)<n} d_Y(\tilde x, \tilde z)_{\check T}\right).$$
When $n=1$ the last term on the right is zero. Since $\check T \le R$ we have $d_R(\tilde x, \tilde y) \le d_{\check T}(\tilde x, \tilde y)$ and the result follows.
\end{proof}

\subsubsection{Thick distance formula}

Combining the Masur-Minsky distance formula (Theorem \ref{MM.formula})  with Corollary \ref{cor.thick_distance} we have our thick distance formula.
\begin{thm}[Thick distance formula]\label{thick.MM}
Let $\tilde x$ be a filling collection on $\Sigma$. Then for $T,R$ sufficiently large
with $R \le \check T$, there exist $C_0,C_1$ such that  for all $g \in MCG(\Sigma)$ 
$$|g| \le  C_0 \sum_{Y\subseteq \Sigma} d^{T,R}_Y(\tilde x, g(\tilde x))+C_1$$
\end{thm}

When we apply this result we will assume that $R = \check T$ and to simplify notation we set
$$d^T_Y(\tilde x, \tilde y) = d^{T, \check T}_Y(\tilde x, \tilde y).$$

\subsection{Separability in the mapping class group}
Let $\psi\in MCG(\Sigma)$ be a pseudo-Anosov. There are various equivalent characterizations, two of which are useful for us.
\begin{itemize}
\item $\psi$ has positive stable translation length on $\C(\Sigma)$.

\item $\psi$ has positive translation length on the Teichm\"uller
  space $\T(\Sigma)$ with a unique invariant axis.
\end{itemize}

If $\psi$ is a pseudo-Anosov then the orbit of any curve in $\C(\Sigma)$ will extend to a $\psi$-invariant quasi-geodesic $\gamma$ and any two such invariant quasi-geodesics will be a bounded Hausdorff distance from each other. The {\em elementary closure}, $EC(\psi)$ is the subgroup of elements $\phi \in MCG(\Sigma)$ such that $\gamma$ and $\phi(\gamma)$ are a bounded Hausdorff distance. Everything that commutes with $\psi$ is contained in $EC(\psi)$ (including powers and roots) but there may be other elements.

The following is well known.
\begin{lemma}\label{virtually cyclic}
If $\psi \in MCG(\Sigma)$ is pseudo-Anosov then $EC(\psi)$ is virtually cyclic.
\end{lemma}

\begin{proof}
We use the second characterization of a pseudo-Anosov. Namely the
action of $\psi$ on the Teichm\"uller space $\T(\Sigma)$ has a unique
axis and the subgroup $EC(\psi)$ will preserve the axis and fix its
endpoints at infinity. Translation length along the axis will define a
homomorphism to $\R$ with discrete image. The subgroup of $EC(\psi)$
of elements with translation length zero will fix the axis pointwise
and will therefore be finite since the stabilizer of any element in
$\T(\Sigma)$ will be finite. In particular there is a surjective map
of $EC(\psi)$ to $\ZZ$ with finite kernel. The lemma follows.
\end{proof}

We note that one can also prove this lemma by applying  a general result
(\cite[Lemma 6.5]{DGO})
which says that if a group $G$ acts on a hyperbolic space,
then the elementary closure (with respect to this action)
of an element $g\in G$ is virtually cyclic if $g$ 
is ``WPD'' (weak proper discontinuous).

Let $Y\subset \Sigma$ be a proper subsurface and let $MCG(\Sigma;Y)$ be the subgroup of the mapping class group that preserves $Y$. If $Y$ is non-annular let $\bar Y$ be the surface obtained by collapsing the components of $\partial Y$ to marked points. There is a natural homomorphism
$$MCG(\Sigma;Y) \to MCG(\bar Y).$$
The kernel of this homomorphism are mapping classes that can be represented by homeomorphisms that are the identity on $Y$. Furthermore every mapping class in the image of the homomorphism is the image of a mapping class that is the identity on the complement of $Y$ and these two types of mapping classes commute.

Given a $\psi\in MCG(\Sigma;Y)$ we say that $\psi$ is {\em pseudo-Anosov on $Y$} if its image in $MCG(\bar Y)$ is pseudo-Anosov. The {\em elementary closure with respect to $Y$}, $EC(\psi; Y)$, is the subgroup of elements $\phi \in MCG(\Sigma;Y)$ whose image in $MCG(\bar Y)$ is contained in $EC(\bar \psi)$. The image $\overline{EC}(\psi; Y)$ of $EC(\psi; Y)$ in $MCG(\bar Y)$ is a subgroup of $EC(\bar\psi)$. Note that image of $\overline{EC}(\psi;Y)$ will be infinite but it may be a proper subgroup of $EC(\bar\psi)$. In particular, $\overline{EC}(\psi;Y)$ contains an infinite cyclic group.

\begin{lemma}\label{MCG centralizer}
Let $Y$ be a non-annular subsurface and assume that $\psi\in MCG(\Sigma;Y)$ with image $\bar\psi$ in $MCG(\bar \psi)$ pseudo-Anosov. Then
$EC(\psi;Y)$ is a centralizer in $MCG(\Sigma)$.
\end{lemma}

\begin{proof}
Choose $\phi\in MCG(\Sigma;Y)$ such that $\phi$ is the identity on the complement of $Y$ and its image $\bar \phi$ in $MCG(\bar Y)$ is a primitive element of infinite order in $\overline{EC}(\psi;Y)$. By Lemma \ref{virtually cyclic}, $\overline{EC}(\psi;Y)$ is virtually cyclic so there is a short exact sequence
$$1\to F\to\overline{EC}(\psi;Y)\to \ZZ\to 1$$
where $F$ is finite and the subgroup $\langle \bar\phi\rangle \subset \overline{EC}(\psi;Y)$ surjects onto $\ZZ$. The subgroup $\langle \bar\phi\rangle$ acts on the finite group $F$ by conjugation so there is a $k$ such that conjugation by ${\bar\phi}^k$ is the identity on $F$. That is ${\bar\phi}^k$ commutes with every element of $F$. As every element of $\overline{EC}(\psi;Y)$ is a product of an element of $F$ and a power of $\bar\phi$ this implies that ${\bar\phi}^k$ commutes with every element of $\overline{EC}(\psi;Y)$ and the centralizer of ${\bar\phi}^k$ in $\overline{EC}(\psi;Y)$ is the entire group.

We now claim that the centralizer of $\phi^k$ in $MCG(\Sigma)$ is $EC(\psi;Y)$. Any element that commutes with $\phi^k$ will be contained in $EC(\psi;Y)$ so we only need to show that every element of $EC(\psi;Y)$ commutes with $\phi^k$. We can decompose every element of $EC(\psi;Y)$ as a composition of three elements:
\begin{itemize}
\item a mapping class $\phi_0$ that is the identity on $Y$;

\item a mapping class $\phi_1$ that is the identity on the on the complement of $Y$ and has finite image in $MCG(\bar Y)$; 

\item a power of $\phi$.
\end{itemize}
As $\phi^k$ will commute with both $\phi_0$ and any power of $\phi$ we only need to show that $\phi^k$ commutes with $\phi_1$. The image ${\bar\phi}_1$ of $\phi_1$ in $MCG(\bar Y)$ has finite order so there exists a $\ell$ such that ${\bar\phi}_1^\ell$ is the identity. Therefore $\phi_1^\ell$ has a representative that is the identity on $Y$ and therefore commutes with $\phi^k$. 

The image of the commutator $[\phi_1, \phi^k]$ in $MCG(\bar Y)$ is
$[{\bar\phi}_1, {\bar\phi}^k]$ and is trivial as these two elements
commute. This implies that $[\phi_1, \phi^k]$ has a representative
that is the identity on $Y$. But it also has a representative that is
the identity on the complement of $Y$ (as both $\phi_1$ and $\phi^k$
do). This implies that $[\phi_1, \phi^k]$ is a composition of Dehn
twists in $\partial Y$. In particular it is either trivial or of
infinite order. As $[\phi_1, \phi^k]$ commutes with both $\phi_1$ and
$\phi^k$ a straightforward calculation shows that $[\phi^j_1, \phi^k]
= [\phi_1, \phi^k]^j$ and therefore if $[\phi_1, \phi^k]$ is
non-trivial then it is of infinite order. However, we observed above
that $\phi^\ell_1$ and $\phi^k$ commute and therefore $[\phi_1^\ell,
  \phi^k]$ is trivial. This implies that $\phi_1$ and $\phi^k$
commute, completing the proof.
\end{proof}

We say that a quasi-geodesic $\gamma\subset \C(Y)$ is an {\em axis} if
there is a $\psi\in MCG(\Sigma;Y)$ and $\gamma$ is
$EC(\psi;Y)$-invariant. Every $\psi$ that is pseudo-Anosov on $Y$ has
an axis that can be obtained by taking the $\overline{EC}(\psi;Y)$
translates of a $\psi$-invariant bi-infinite path.

To match the notation from Theorem \ref{thm1} we let $C(\gamma)\subset MCG(\Sigma)$ be the stabilizer of $\gamma$. When $\gamma$ is an axis for $\psi$ we have $C(\gamma) = EC(\psi;Y)$.

By \cite{grossman}, the mapping class group is residually finite. Then Lemma \ref{separability} combined with Lemmas \ref{MCG centralizer}:
\begin{cor}\label{mcg_separable}
If $\gamma \subset \C(Y)$ is an axis then $C(\gamma)$ is is separable.
\end{cor}

Note that to this point we have not discussed the case when $Y$ is an annulus. Here all of $\C(Y)$ is a quasi-geodesic and will play the role of an axis. While it is true that the stabilizer of $\C(Y)$ is separable (\cite{leininger}) we will not use this.

\subsection{Projection axioms for axes in curve graphs}
We begin with our setup for mapping class groups:
\begin{itemize}
\item $\bY$ is a collection of transverse subsurfaces. 

\item $\tilde\A_\bY$ is the collection of all quasi-geodesics in all the
  curve graphs $\C(Y)$, $Y\in\bY$, with uniform
  constants. (We fix uniform quasi-geodesic constants to start with.)
  For each subsurface $Y$ we let $\tilde\A_Y$ be the
  subcollection contained in $\C(Y)$.

\item $\A_\bY\subset \tilde\A_\bY$ is a subcollection.
%
\item For $\alpha\in \A_X$ and $\beta\in\A_X$ we define $\pi_\alpha(\beta)$ to be uniformly close (in the Hausdorff metric) to the nearest point projection of $\alpha$ to $\beta$.

\item For $\alpha\in\A_X$ and $\beta\in\A_Z$ when $X\neq Z$ we define $\pi_\alpha(\beta)$ to be uniformly close to the nearest point projection of $\pi_X(Z)$ to $\alpha$.
\end{itemize}

For mapping class groups the coarse constants are the quasi-geodesic constants above along with the  projection constant and BGIT constant from Theorem \ref{axioms-hold} and the hyperbolicity constant for curve graphs.
\begin{thm}\label{surface axioms}
Fix $\theta>0$. Then there exists $\chi>0$, depending only on the coarse constants and $\theta$, such that if $\diam \pi_\alpha(\beta)\le \theta$ whenever $\alpha$ and $\beta$ are distinct elements in the same $\A_Y$ then $(\A, \{\pi_\gamma\})$ satisfies the projection axioms with projection constant $\chi$.
\end{thm}

\begin{proof}
If $\gamma_0,\gamma_1\in \A_Y$ then (P0) holds by assumption. If $\gamma_0$ and $\gamma_1$ are in distinct $\A_Y$ then (P0) holds by Theorem \ref{axioms-hold}.

For the remaining two axioms we first observe:
\begin{enumerate}[($*$)]
\item If $\gamma_0,\gamma_1,\gamma_2$ are in $\A_{Y_0}, \A_{Y_1}, \A_{Y_2}$ and the $Y_0$ and $Y_2$ are distinct from $Y_1$ then $d_{\gamma_1}(\gamma_0,\gamma_2)$ is coarsely bounded above by $d_{Y_1}(Y_0, Y_2)$.
\end{enumerate}
This follows from the fact that nearest point projections in $\delta$-hyperbolic spaces are coarsely Lipschitz. 

When all three $Y_i$ are distinct then (P1) follows directly from $(*)$ and Theorem \ref{axioms-hold}. If the three $Y_i$ are all equal then (P1) follows from Proposition \ref{hyperbolic axioms}. The last case is when $Y_0 = Y_1$ but they are distinct from $Y_2$. In this case $d_{\gamma_2}(\gamma_0, \gamma_1)$ will be uniformly bounded $(*)$. Applying Proposition \ref{hyperbolic axioms} to $\gamma_0, \gamma_1$ and $\pi_{Y_0=Y_1}(Y_2)$ we see that at most one of $d_{\gamma_0}(\gamma_1, \gamma_2)$ and $d_{\gamma_1}(\gamma_0,\gamma_2)$ are large proving (P1) in this final case.

Now we prove (P2). Fix $\alpha\in\A_X$ and $\beta\in \A_Z$. If $Y$ is distinct from $X$ and $Z$  and $\gamma\in\A_Y$ with $d_\gamma(\alpha, \beta)$ large then by $(*)$ we have that $d_Y(X,Z)$ is large. Therefore Theorem \ref{axioms-hold} implies that there are finitely many $Y$ such that $\A_Y$ contains a $\gamma$ with $d_\gamma(\alpha, \beta)$ large. Applying Proposition \ref{hyperbolic axioms} to the collection of quasi-convex sets $\A_Y\cup\{\pi_Y(X), \pi_Y(Z)\}$ we see that in each such $Y$ there are finitely many $\gamma\in\A_Y$ with $d_\gamma(\alpha, \beta)$ large. Similarly we get finitely many $\gamma\in \A_X$ with $d_\gamma(\alpha, \beta)$ large by applying Proposition \ref{hyperbolic axioms} to $\A_X\cup \{\pi_X(Z)\}$ if $X$ and $Z$ are distinct or simply to $A_X$ if $X=Z$. This proves (P2).
\end{proof}

Next we prove the version of Proposition \ref{main estimate} that we need for the mapping class group.
\begin{prop}\label{curve estimate}
Fix  $K>0$. Then there exists an $T>0$ depending only on the coarse constants and $K$ such that the following holds.
Assume that
\begin{itemize}
\item $\tilde x$, $\tilde y$ are filling collections in $\C(\Sigma)$;

\item $\hat x=(x_1,x_2,\cdots,x_n)$ and $\hat y=(y_1,y_2,\cdots,y_n)$
  are $n$-tuples of curves with each $x_i$ and $y_i$ lying in curve graphs $\C(X_i)$ and $\C(Y_i)$ with $X_i$ and $Y_i$ in $\bY$; 

\item any tight geodesic segment in $\C(Y)$, with $Y\in\bY$, that is $\hat T$-thick
  and of length $\check T$ is contained in some $\gamma \in\A_Y$;

\item $\A_\bY$ is partitioned into $\A_1\sqcup \dots \sqcup \A_n$.
\end{itemize}
Then
$$\sum_{Y\in\bY} d_Y^{T}(\tilde x, \tilde y) \le 2\sum_i\sum_{\gamma\in\A_i} d_\gamma(x_i, y_i)_K  + \frac{2T}\xi\sum_{Y\in\bY} \left(d_Y(\tilde x,  \hat x)_\xi + d_Y(\tilde y,\hat y)_\xi\right).$$
\end{prop}

\begin{proof}
We choose $T$ in a way similar to the choice of $L$ in the proof of Proposition \ref{main estimate}:
Fix $x,x',y,y'\in\C(Y)$ 
such that $d_{\C(Y)}(x, x') \le d_Y(\tilde x, \hat x)$,  $d_{\C(Y)}(y,y') \le d_Y(\tilde y, \hat y)$ and
$$d_{\C(Y)}(x,y) > d_Y(\tilde x, \hat x) + d_Y(\tilde y, \hat y).$$
and let $\tilde\alpha$ be the subsegment of a geodesic from $x$ to $y$
where the $\max\{d_Y(\tilde x, \hat x) - \xi,0\}$ and $\max\{d_Y(\tilde y,\hat y)-\xi,0\}$
neighborhoods of each endpoint have been removed. Note that
$d_Y(\tilde x,\hat x) \le \xi$ and $d_Y(\tilde y, \hat y)\le \xi$ for
all but finitely many $Y$. Then the nearest point projection of $x$
and $x'$ to any subsegment of $\tilde\alpha$ will be in a uniformly
bounded neighborhood of the endpoint closest to $x$ with a similar
statement for the projection of $y$ and $y'$.
As in Proposition \ref{main estimate}, for any $\beta\in\A_\bY$ we have a uniform upper bound on 
$$\diam(\tilde\alpha\cap\beta)-d_\beta(x',y')$$
and we can choose $T$ such that if $\tilde\alpha\cap \beta$ contains a path of length $\check T$ then
\begin{itemize}
\item $\diam(\tilde\alpha\cap\beta) \le 2d_\beta( x', y')$ and
\item $\diam(\tilde\alpha\cap \beta) \ge 2K$.
\end{itemize}

For each subsurface $Y$ where $d^{T}_Y(\tilde x,\tilde y)>0$ we let $\alpha_Y$ be a tight geodesic between $x\in\pi_Y(\tilde x)$ and $y\in\pi_Y(\tilde y)$ that realizes $d^{T}_Y(\tilde x,\tilde y)$. In particular there are disjoint subsegments $\alpha^Y_1, \dots, \alpha^Y_{n_Y}$ of $\alpha_Y$ such that endpoints of $\alpha^Y_i$ are $T$-thick and $d^T_Y(\tilde x, \tilde y)$ is the sum of the lengths of the $\alpha^Y_i$. By Lemma \ref{tight-bound}, each subsegment of length $\check T$ in each $\alpha^Y_i$ will have endpoints that are $\hat T$-thick and therefore each such subsegment will be contained in $\gamma\in \A_\bY$. Let $\tilde\alpha_Y$ be obtained by removing the $d_Y(\tilde x,  \hat x) - \xi$ and $d_Y(\tilde y, \hat y)-\xi$ neighborhoods of each endpoint of $\alpha_Y$. By the above estimate if $\gamma \in \A_i$ and $\tilde\alpha_Y \cap \gamma$ contains a path of length $\check T$ then
$$K\le \frac12\diam(\tilde\alpha_Y\cap \gamma) \le d_\gamma(x_i, y_i).$$
Let $\mathcal I$ be the indices $i$ such that $\diam(\tilde\alpha_Y\cap \alpha^Y_i) \ge \check T$. Then
$$\sum_{i\in\mathcal I}\diam(\tilde\alpha_Y\cap\alpha^Y_i) \le 2\sum_j\sum_{\gamma\in\A_j\cap\A_Y} d_\gamma(x_j, y_j)_K.$$
To complete the proof we will show that
$$d^T_Y(\tilde x, \tilde y) - \sum_{i\in\mathcal I}\diam(\tilde\alpha_Y\cap\alpha^Y_i)\le \frac{2T}{\xi} \left(d_Y(\tilde x,\hat x)_\xi + d_Y(\tilde y, \hat y)_\xi\right).$$
If $d_Y(\tilde x, \hat x)_\xi = d_Y(\tilde y,\hat y)_\xi=0$  then $\alpha_Y=\tilde\alpha_Y$ and the two terms are equal and the difference is zero proving the bound in this case. If both $d_Y(\tilde x, \hat x) \ge \xi$ and $d_Y(\tilde y,\hat y) \ge \xi$ then
$$d^T_Y(\tilde x, \tilde y) - \sum_{i\in\mathcal I}\diam(\tilde\alpha_Y\cap\alpha^Y_i)\le d_Y(\tilde x, \hat x) + d_Y(\tilde y,\hat y) -2\xi + 2\check T.$$
Note that $2\check T$ on the right comes from the fact that there may be two (but no more) $\alpha^Y_i$ that intersect $\tilde\alpha_Y$ in a segment of length $<\check T$. Similarly if only $d_Y(\tilde x, \hat x) \ge \xi$ then
$$d^T_Y(\tilde x, \tilde y) - \sum_{i\in\mathcal I}\diam(\tilde\alpha_Y\cap\alpha^Y_i)\le d_Y(\tilde x, \hat x) -\xi + \check T.$$
If only $d_Y(\tilde y, \hat y) \ge \xi$ the roles of $x$ and $y$ are swapped in the above inequality. If we combine these bounds with the fact that if $C\ge \xi$ then
$$C-\xi + \check T \le \frac{2T}\xi C$$
we get the desired bound in all cases.

The proof is then completed by summing the inequality
$$d^T_Y(\tilde x,\tilde y) \le 2\sum_j\sum_{\gamma\in \A_j\cap\A_Y} d_\gamma(x_j, y_j)_K + \frac{2T}{\xi}\left(d_Y(\tilde x, \hat x)_\xi + d_Y(\tilde y, \hat y)_\xi\right)$$
over all $Y\in\bY$.
\end{proof}

\subsection{Axes}

Here is our replacement for Theorem \ref{quasi-axes} in the setting of mapping class groups.
\begin{thm}\label{constructing pAs}
There exists a $MCG(\Sigma)$-invariant collection of uniform quasi-geodesics $\tilde\A$ with $\tilde\A_Y$ the subcollection of $\tilde\A$ contained in $\C(Y)$ such that
\begin{itemize}
\item if $Y$ is an annulus then $\tilde\A_Y= \{\C(Y)\}$;

\item if $Y$ is non-annular then $\tilde\A_Y$ is a collection of axes and every geodesic segment $\sigma \in \C(Y)$ of length $\ge 3$ is contained in some $\gamma\in \tilde \A_Y$.
\end{itemize}
\end{thm}

We need a preliminary lemma.
\begin{lemma}\label{local big}
Let $z_0, \dots, z_n$ be a collection of curves in $\cC(Y)$ such that 
$$d_{z_i}(z_{i-1},z_{i+1}) \ge 3\xi$$
for $i=1, \dots, n-1$. Then
$$d_Y(z_0, z_n) \ge \sum_{i=1}^{n} d_{\C(\Sigma)}(z_{i-1}, z_{i}) + 2-2n.$$

\end{lemma}

\begin{proof}
We first show that if $0<i<n$ then
$$|d_{z_i}(z_0, z_n) - d_{z_i}(z_{i-1}, z_{i+1})| \le 2\xi$$
and therefore
$$d_{z_i}(z_0, z_n) \ge d_{z_i}(z_{i-1}, z_{i+1}) -2\xi \ge \xi.$$
We induct on $n$ where the base case is when $n=2$. Since $i< n$ we have
$$d_{z_{i-1}}(z_0, z_i) \ge \xi$$
so by (P1)
$$d_{z_i}(z_0, z_{i-1}) \le \xi.$$
Similarly
$$d_{z_i}(z_{i+1}, z_n) \le \xi$$
and the desired bounds follow from the triangle inequality. 
We note that this also proves that all of the $z_i$ intersect.

We now prove the distance estimate via induction. The base case is when $n=1$ and the inequality is an equality by observation. Now assume the estimate holds for $k$. By the BGIT any geodesic from $z_0$ to $z_k$ must pass within one of $z_k$. Then by the triangle inequality
$$d_{\C(Y)}(z_0,z_{k+1}) \ge d_{\C(Y)}(z_0, z_k) + d_{\C(Y)}(z_k,z_{k+1}) - 2$$
and the bound follows.
\end{proof}

\begin{proof}[Proof of Theorem \ref{constructing pAs}]
Let $\sigma$ be a geodesic segment of length $\ge 3$ in some $\C(Y)$ with $Y$ non-annular.
Let $x$ and $y$ be the endpoints of $\sigma$. We show that for a sufficiently large positive integer $n$ the composition of Dehn twists $\psi = D_y^n D_x^n$ is a pseudo-Anosov on $Y$ with an axis (with uniform constants) that contains $\sigma$. Let $z_{2k} = \psi^k(x)$ and $z_{2k+1} = \psi^k(y)$. Let 
$$\gamma = \cup_k \psi^k(\sigma\cup D^n_x(\sigma)).$$
This is a $\psi$-equivariant path where the $z_k$ appear in the order given by their indices and the path is geodesic between each $z_k$ and $z_{k+1}$.
Note that
$$d_{z_i}(z_{i-1}, z_{i+1}) = d_{z_i}(z_{i-1}, D^n_{z_{i}}(z_{i-1})) \sim n$$
so when $n$ is sufficiently large the $z_i$ satisfies the conditions of Lemma \ref{local big}. The estimate there implies that $\gamma$ is a quasi-geodesic with uniform constants. To get a quasi-geodesic that is $EC(\psi;Y)$-invariant we take $\gamma_\sigma$ to be the $EC(\psi;Y)$ translates of $\gamma$. This is a collection (in fact a finite collection) of quasi-geodesics with uniform constants that are all in a bounded Hausdorff distance of each other. Therefore $\gamma_\sigma$ is uniformly quasi-isometric to $\ZZ$.

The collection of geodesic segments in curve graphs $\C(Y)$ with $Y$ non-annular is $MCG(\Sigma)$-invariant.  We choose a representative in each $MCG(\Sigma)$-orbit and apply the  above construction and then take the $MCG(\Sigma)$-orbit of this collection of axes. Finally we add all of the curve graphs $\C(Y)$ with $Y$ annular to form $\tilde\A$.
\end{proof}

\subsection{Constants}\label{constants}
As in the proof of Theorem \ref{thm1} we are now ready to fix constants. Fix $\theta>0$ as in Proposition \ref{acylind_finite} such that for any axis $\gamma \in \tilde\A_Y$ there are only finitely many double cosets in of $C(\gamma)$ in $MCG(\Sigma;Y)$ whose projection to $\gamma$ is $> \theta$.
Choose $\chi'>0$ to be the projection constant from Theorem \ref{surface axioms} with diameter bound $\theta$. As before we modify the projections so that the strong projection axioms hold with projection constant $\chi$ and let $K=4\chi$ so that the distance formula holds with threshold $K$. When then choose $T$ with respect to $K$ as in Proposition \ref{curve estimate}.

\subsection{Preferred axes}
By Corollary \ref{finite-thick} there are finitely many $\hat T$-thick geodesic segments $\sigma_1, \dots, \sigma_n$ of length $\check T$ such that every other $\hat T$-thick geodesic segment of length $\check T$ is contained in the $MCG(\Sigma)$-orbit of one of the $\sigma_i$. By Theorem \ref{constructing pAs}, for each $\sigma_i$ there exists a $\gamma_i \in \tilde \A$ such that $\sigma_i\subset \gamma_i$. Let $\A$ be the $MCG(\Sigma)$-orbits of the $\gamma_i$.

\subsection{Coloring $\A$}
We can choose a subgroup $G< MCG(\Sigma)$ such that the $G$-orbit of a subsurface is a transverse collection. 
When $\Sigma$ is a closed surface this is Theorem \cite[Lemma 5.7]{bbf}.
When $\Sigma$ has punctures, the statement follows by blowing up the punctures to boundary components to obtain a surface $\tilde \Sigma$
and doubling to obtain the surface $D\tilde \Sigma$. Then $MCG(\tilde \Sigma)$
is a subgroup of $MCG(D\tilde \Sigma)$ and we have the desired finite index subgroup by taking the intersection with the one in $MCG(D\tilde \Sigma)$.
Finally we project this subgroup to $MCG(\Sigma)$.

By Corollary \ref{mcg_separable} the stabilizer $C(\gamma_i)$ is separable in $MCG(\Sigma)$ and therefore in $G$. By our choice of $\theta$ from Proposition \ref{acylind_finite} we can use the separability $C(\gamma_i)$ to find a finite index subgroup $H_i$ of $G$ such that if $h \in H_i \cap MCG(\Sigma;Y)$ then $\diam \pi_{\gamma_i}(h(\gamma_i)) < \theta$.
Let
$$H = H_1\cap \dots \cap H_n$$
and add axes $\gamma_{n+1}, \dots, \gamma_m$ so that we  have one axis in $\A$ for each $H$-orbit. Let $\A_i$ be the $H$-orbit of $\gamma_i$. The $\A_i$ will partition $\A$ and each one will satisfy the conditions of Theorem \ref{surface axioms}.

\subsection{Product of quasi-trees $\mathcal X$}
Again we follow the proof of Theorem \ref{thm1}. For each $\A_i$ we have the quasi-tree $\C_K(\A_i)$. Choose a filling collection $\tilde x$ and an $m$-tuple of curves $\hat x= \{x_1, \dots, x_m\}$ such that each $x_i$ lies in some axis in $\A_i$. We let
$$\mathcal X = \prod_{i=1}^m \C_K(\A_i)$$
with the $\ell^1$-metric. We will show that $H$ quasi-isometrically embeds in $\mathcal X$. By Section \ref{induction} this implies that $MCG(\Sigma)$ quasi-isometrically embeds in a finite product of quasi-trees. The $H$-orbit of $\hat x$ in $\mathcal X$ gives a Lipschitz embedding of $H$ in $\mathcal X$. For the lower bound we have
$$\frac14\sum_i\sum_{\gamma\in\A_i} d_\gamma(x_i, h(x_i))_K \le d_{\mathcal X}(\tilde x, h(\tilde x))$$
for all $h\in H$ by the distance formula. By Proposition \ref{curve estimate} we have
$$\sum_{Y\subseteq\Sigma} d^T_Y(\tilde x, h(\tilde x)) \le 2\sum_i\sum_{\gamma\in\A_i} d_\gamma(x_i, h(x_i))_K + \frac{2T}\xi \sum_{Y\subseteq\Sigma}(d_Y(\tilde x, \hat x)_\xi + d_Y(h(\tilde x), h(\hat x))_\xi).$$
The last term on the right is finite since $d_Y(\tilde x, \hat x)<\xi$ for all but finitely many $Y$ and is  independent of $h$ since
$$d_Y(\tilde x, \hat x)_\xi = d_{h(Y)}(h(\tilde x), h(\hat x))_\xi$$
by the group equivariance of the projections. The thick distance formula (Theorem \ref{thick.MM}) then gives a linear lower bound on the left hand side of the inequality 
in terms of the word length $|h|$. This completes the proof Theorem \ref{thm2}.

\bibliographystyle{amsalpha}
\bibliography{ref}

\end{document}